\documentclass[12pt,a4paper]{article}

\usepackage[usenames]{color}
\usepackage{hyperref}

\usepackage[centertags]{amsmath}
\usepackage{amsfonts}
\usepackage{amssymb}
\usepackage{amsthm}
\usepackage{latexsym}
\usepackage{newlfont}
\usepackage{verbatim}
\usepackage[english]{babel}
\usepackage{texdraw}
\usepackage{graphicx}
\usepackage[matrix,arrow]{xy}

\newtheorem{thm}{Theorem}[section]
\newtheorem{cor}[thm]{Corollary}
\newtheorem{lem}[thm]{Lemma}
\newtheorem{prop}[thm]{Proposition}

\newtheorem{defn}[thm]{Definition}
\theoremstyle{remark}

\newcommand{\D}{d}

\newcommand{\ml}{\textbf}

\def\proof{\trivlist \item[\hskip \labelsep{\bf Proof.\ }]}

\newtheorem{lemma}{Lemma}[section]

\newtheorem{example}[lemma]{Example}
\newtheorem{remark}[lemma]{Remark}

\newcommand{\bl}{\begin{lem}}
\newcommand{\el}{\end{lem}}
\newcommand{\bt}{\begin{thm}}
\newcommand{\et}{\end{thm}}
\newcommand{\bc}{\begin{cor}}
\newcommand{\ec}{\end{cor}}
\newcommand{\bp}{\begin{proof}}
\newcommand{\ep}{\end{proof}}
\newcommand{\bpr}{\begin{prop}}
\newcommand{\epr}{\end{prop}}
\newcommand{\brem}{\begin{remark} }
\newcommand{\erem}{\end{remark}}
\newcommand{\bd}{\begin{defn} \em}
\newcommand{\ed}{\end{defn}}
\newcommand{\bex}{\begin{example}
}
\newcommand{\eex}{\end{example}}

\newcommand{\bi}{\begin{itemize}
  }
\newcommand{\ei}{\end{itemize}}
\newcommand{\ben}{\begin{enumerate} }
\newcommand{\een}{\end{enumerate} }

\newenvironment{enumr}{

\begin{enumerate}     }{\end{enumerate}

}

\newcommand{\al}[1]{\forall #1\:}

\newlength{\hilflh}

\renewcommand{\emptyset}{\varnothing}

\newcommand{\cL}{{\mathcal L}}

\newcommand{\cP}{{\mathcal P}}

\newcommand{\cO}{{\mathcal O}}
\newcommand{\cC}{{\mathcal C}}

\newcommand{\ga}{\alpha}
\newcommand{\gb}{\beta}
\newcommand{\gd}{\delta}
\renewcommand{\ge}{\varepsilon}
\newcommand{\gl}{\lambda}

\newcommand{\gs}{\sigma}
\newcommand{\gy}{\gamma}
\newcommand{\gw}{\omega}

\renewcommand{\phi}{\varphi}

\newcommand{\imp}{\rightarrow}

\newcommand{\GL}{\mathbf{GL}}
\newcommand{\GLP}{\mathbf{GLP}}

\newcommand{\ZFC}{\mathsf{ZFC}}

\newcommand{\rst}{\upharpoonright}

\newcommand{\la}{\langle}
\newcommand{\ra}{\rangle}

\renewcommand{\models}{\vDash}      

\newcommand{\nmodels}{\nvDash}

\newcommand{\Imp}{\Rightarrow}

\newcommand{\Var}{\mathrm{Var}}

\newcommand{\On}{\mathrm{On}}

\newcommand{\Lim}{\mathrm{Lim}}
\newcommand{\Suc}{\mathrm{Suc}}

\newcommand{\Log}{\mathrm{Log}}
\newcommand{\tto}{\twoheadrightarrow}
\newcommand{\iso}{\textit{iso}}
\newcommand{\J}{\mathbf{J}}
\newcommand{\exs}[1]{\exists #1\:}

\newcommand{\eop}{$\boxtimes$ \protect\par \addvspace{\topsep}}

\begin{document}

\author{Lev Beklemishev\thanks{Supported by the Russian Foundation for Basic Research (RFBR), Russian Presidential Council for Support of Leading Scientific Schools, and the Swiss--Russian cooperation project STCP--CH--RU ``Computational proof theory.''} \\ Steklov Institute of Mathematics, Moscow \and David Gabelaia \\ Razmadze Institute of Mathematics, Tbilisi}

\title{Topological completeness of the provability logic GLP}

\maketitle

\begin{abstract}
Provability logic $\GLP$ is well-known to be incomplete w.r.t.\ Kripke semantics. A natural topological semantics of $\GLP$ interprets modalities as derivative operators of a polytopological space. Such spaces satisfying all the axioms of $\GLP$ are called GLP-spaces. We develop some constructions to build nontrivial GLP-spaces and show that $\GLP$ is complete w.r.t.\ the class of all GLP-spaces.

\end{abstract}

\section{Introduction}

This paper continues the study of topological semantics of an important polymodal provability logic $\GLP$ initiated in \cite{BBI09,Bek10a}. This system, introduced by Japaridze \cite{Dzh86,Dzh88}, describes in the style of provability logic all the universally valid schemata for the reflection principles of restricted logical complexity in arithmetic. Thus, it is complete with respect to a very natural kind of proof-theoretic  semantics.

The logic $\GLP$ has been extensively studied in the early 1990s by Ignatiev and Boolos who simplified and extended Japaridze's work (see \cite{Boo93}). More recently, interesting applications of $\GLP$ have been found in proof theory and ordinal analysis of arithmetic. In particular, $\GLP$ gives rise to a natural system of ordinal notations for the ordinal $\ge_0$. Based on the use of $\GLP$, the first author of this paper gave a proof-theoretic analysis of Peano arithmetic, which stimulated further interest towards $\GLP$ (see \cite{Bek04,Bek05} for a detailed survey).

The main obstacle in the study of $\GLP$ is that it is incomplete w.r.t.\ any class of Kripke frames. However, a more general topological semantics for the G\"odel--L\"ob provability logic $\GL$ has been known since the work of Simmons~\cite{Sim75} and Esakia \cite{Esa81}. In the sense of this semantics, the diamond modality is interpreted as the topological derivative operator acting on a scattered topological space. The idea to extend this approach to the polymodal logic $\GLP$ comes quite naturally.\footnote{Leo Esakia raised this question several times in conversations with the first author.}

The language of $\GLP$ has denumerably many modalities each of which individually behaves like the one of $\GL$ and can therefore be interpreted as a derivative operator of a polytopological space $(X,\tau_0,\tau_1,\dots)$. The additional axioms of $\GLP$ imply certain dependencies between the scattered topologies $\tau_i$, which lead the authors of \cite{BBI09} to the concept of \emph{GLP-space}. Thus, GLP-spaces provide an adequate topological semantics for $\GLP$.

The question of completeness of $\GLP$ w.r.t.\ this semantics turned out to be more difficult. The main contribution of \cite{BBI09} was to show that the fragment of $\GLP$ with only two modalities was topologically complete. However, already for the fragment with three modalities the question remained open.
The present paper answers this question positively for the language with infinitely many modalities and shows that $\GLP$ is complete w.r.t.\ the semantics of GLP-spaces.

\section{Preliminaries}

$\GLP$ is a propositional modal logic formulated in a language with infinitely many
modalities $[0]$, $[1]$, $[2]$, \dots. As usual, $\la n\ra\phi$ stands for $\neg[n]\neg\phi$, and $\bot$ is the logical constant `false'.
$\GLP$ is given by the following axiom schemata and inference rules.
\begin{description}
\item[Axioms:]
\begin{enumr}
\item[]
\item Boolean tautologies; \item $[n](\phi\imp\psi)\imp
([n]\phi\imp [n]\psi)$; \item $[n]([n]\phi\imp \phi)\imp [n]\phi$ (L\"ob's axiom);
\item $[m]\phi\imp [n]\phi$, for $m<n$;
\item $\la m\ra\phi \imp [n]\la m\ra\phi$, for $m<n$.
\end{enumr}
\item[Rules:]
\begin{enumr}
\item[]
\item $\vdash\phi,\ \vdash\phi\to\psi\ \Imp\ \vdash\psi$ (modus ponens);
\item $\vdash\phi\ \Imp\ \vdash [n]\phi$, for each $n\in\omega$ (necessitation).
\end{enumr}
\end{description}

In other words, for each modality, $\GLP$ contains the axioms and inference rules of the G\"{o}del-L\"{o}b Logic $\GL$. Axioms (iv) and (v) relate different modalities to one another.

\emph{Neighborhood semantics} for modal logic can be seen both as a generalization of Kripke semantics and as a particular kind of algebraic semantics. Let $X$ be a nonempty set and let $\delta_n:\cP(X)\to \cP(X)$, for each $n\in\gw$, be some unary operators acting on the boolean algebra of all subsets of $X$. Such a structure $X$ will be called a \emph{neighborhood frame}. A \emph{valuation} on $X$ is a map $v:\Var\to \cP(X)$ from the set of propositional variables to the powerset of $X$, which is extended to all formulas in the language of $\GLP$ as follows:
\bi
\item $v(\phi\lor \psi)=v(\phi)\cup v(\psi)$, $v(\neg \phi)= X\setminus v(\phi)$, $v(\bot)=\emptyset$,
\item $v(\la n\ra\phi) = \delta_n(v(\phi))$, $v([n]\phi) = \tilde{\delta}_n(v(\phi)),$
where $\tilde{\delta}_n(A):=X\setminus\delta_n(X\setminus A)$, for any $A\subseteq X$.
\ei
A formula $\phi$ is \emph{valid in $X$}, denoted $X\models\phi$, if $v(\phi)=X$ for all $v$. The \emph{logic of $X$} is the set $\Log(X)$ of all formulas valid in $X$.

Next we observe that any neighborhood frame of $\GLP$ is, essentially, a polytopological space, in which all operators $\delta_n$ can be interpreted as the derived set operators.

Suppose $(X,\tau)$ is a topological space. The \emph{derived set operator} on $X$ is the map $d_\tau:\cP(X)\to\cP(X)$ associating with each $A\subseteq X$ its set of limit points, denoted $d_\tau(A)$. In other words, $x\in d_\tau(A)$ iff every open neighborhood of $x$ contains a point $y\neq x$ such that $y\in A$. We shall write $dA$ for $d_\tau(A)$ whenever the topology $\tau$ is given from the context.

A topological space $(X,\tau)$ is called \emph{scattered} if every nonempty subspace $A\subseteq X$ has an isolated point. A polytopological space $(X,\tau_0,\tau_1,\dots)$ is called a \emph{GLP-space} (cf. \cite{BBI09}) if the following conditions hold, for each $n<\gw$:
\bi
\item $\tau_n$ is scattered;
\item $\tau_n\subseteq \tau_{n+1}$;
\item $d_{\tau_n}(A)$ is $\tau_{n+1}$-open, for each $A\subseteq X$.
\ei
This concept is justified by the basic observation that GLP-spaces are equivalent to the neighborhood frames validating all the axioms of $\GLP$. Thus, to each $\GLP$-space we associate a neighborhood frame $(X,d_0,d_1,\dots)$ where $d_n=d_{\tau_n}$, for each $n<\gw$. Then the following proposition holds.

\bpr \label{neighbor}
\begin{enumr}
\item If $(X,\tau_0,\tau_1,\dots)$ is a GLP-space, then in the associated neighborhood frame all the theorems of $\GLP$ are valid: $(X,d_0,d_1,\dots)\models \GLP$.
\item
Suppose $(X,\delta_0,\delta_1,\dots)$ is a neighborhood frame such that $X\models \GLP$. Then there are naturally defined topologies $\tau_0$, $\tau_1$, \dots on $X$ such that $\delta_n=d_{\tau_n}$, for each $n<\gw$. Moreover, $(X,\tau_0,\tau_1,\dots)$ is a GLP-space.
\end{enumr}
\epr

A proof of this proposition builds upon the ideas of H.~Simmons~\cite{Sim75} and L.~Esakia~\cite{Esa81,Esa03}, which by now have become almost folklore, but it is somewhat lengthy. For the reader's convenience we give this proof in the Appendix.

By Proposition~\ref{neighbor}, the study of neighborhood semantics for $\GLP$ becomes the study of GLP-spaces. Since $\GLP$ is well-known to be incomplete w.r.t.\ any class of Kripke frames the following question naturally arises:
 \bi \item
Is $\GLP$ complete w.r.t.\ neighborhood semantics?
\ei
In other words, we ask whether there is a suitable class of neighborhood frames $\cC$ such that any formula is valid in all frames in $\cC$ iff it is provable in $\GLP$. Equivalently, this problem was stated in \cite{BBI09} as the question whether $\GLP$ is the logic of the class of all GLP-spaces.

This question was positively answered for the language with only two modalities in \cite{BBI09}. However, for the case of three or more modalities even a more basic problem was open:

\bi \item Is there a $\GLP$-space in which all the topologies are non-discrete?
\ei

Some difficulties surrounding these problems are exposed in the papers \cite{BBI09,Bek10a,Bek09a}. Given a scattered space $(X,\tau)$ we can define a new topology $\tau^+$ on $X$ as the coarsest topology containing $\tau\cup\{d_\tau(A):A\subseteq X\}$. Then $(X,\tau,\tau^+,\tau^{++},\dots)$ becomes a GLP-space which we call a GLP-space \emph{naturally generated from $(X,\tau)$}.

As a fundamental example, one can consider the class of GLP-spaces naturally generated from the standard order topology $\tau_<$ on the ordinals. We call them \emph{ordinal GLP-spaces}. Quite unexpectedly, these spaces turned out to have some deep relations with set theory, in particular, with stationary reflection. For example, it can be shown that the first limit point of $\tau_<^+$ is the cardinal $\aleph_1$, whereas the first limit point of $\tau_<^{++}$ is the so-called \emph{doubly reflecting cardinal}. The existence of this (relatively weak) large cardinal is, however, independent from the axioms of $\ZFC$. Thus, it is independent from $\ZFC$ whether $\tau^{++}$ is discrete on any ordinal GLP-space.

In spite of the above, the present paper gives positive answers to both questions formulated above while firmly standing on the grounds of $\ZFC$. This is achieved by developing new topological techniques related to the study of maximal rank preserving extensions of scattered topologies. In particular, we introduce a certain class of topologies we call \emph{$\ell$-maximal} and show that they are sufficiently well-behaved w.r.t.\ the operation $\tau\mapsto \tau^+$.

As another ingredient of the topological completeness proof, we introduce an operation on scattered spaces called \emph{$d$-product}. It can be seen as a generalization of the usual multiplication operation on the ordinals (considered as linear orderings) to arbitrary scattered spaces. We think that this operation could be of some interest in its own right.

The paper is organized as follows. In Section 3 we introduce some useful standard notions related to scattered spaces and prove a few facts about the Cantor--Bendixon rank function. Maximal rank preserving and $\ell$-maximal spaces are introduced in Section 4. In Section 5 we show how this techniques allows one to build a non-discrete GLP-space. Section 6 essentially deals with logic and contains a reduction of the topological completeness theorem to some statement of purely topological and combinatorial nature (main lemma). The rest of the paper is devoted to a proof of this lemma. In Section 7 the $d$-product operation is introduced and a few basic properties of this operation are established. Using $d$-products, as well as the techniques of Sections 4 and 5, two basic constructions on GLP-spaces are presented in Section 6. Finally, Section 8 contains a proof of the main lemma.

\section{Scattered spaces, ranks and $d$-maps}

Given a scattered space
$X=(X,\tau)$ one can define a transfinite \emph{Cantor--Bendixon sequence} of closed subsets $\D^\alpha X$ of $X$, for any ordinal $\alpha$, as follows:
\begin{itemize}
\item $\D^0 X=X$;
\quad $\D^{\alpha+1} X=\D(\D^{\alpha} X)$ and
\item $\D^\alpha X=\bigcap\limits_{\beta<\alpha}\D^\beta X$ if $\alpha$ is a limit
ordinal.
\end{itemize}
Since $X$ is a scattered space, $\D^{\alpha+1} X\subset\D^\alpha
X$ is a strict inclusion unless $\D^\alpha X=\emptyset$.
Therefore, from cardinality considerations, for some ordinal
$\ga$ we must have $\D^{\ga} X=\emptyset$. Call the least such
$\ga$ the \emph{Cantor--Bendixon rank} of $X$ and denote it by
$\rho(X)$. The \emph{rank function} $\rho_X:X\to \On$ is defined by
$$\rho_X(x):=\min\{\ga: x\notin d^{\ga+1}(X)\}.$$
Notice that $\rho_X$ maps $X$ onto $\rho_X(X)=\{\ga:\ga<\rho(X)\}$. Also, $\rho_X(x)\geq\ga$ iff $x\in d^\ga X$. We omit the subscript $X$ whenever there is no danger of confusion.

\bex
Let $\Omega$ be an ordinal equipped with its \emph{left topology}, that is, a subset $U\subseteq \Omega$ is open iff $\al{\ga\in U}\al{\gb<\ga} \gb\in U$. Then $\rho(\ga)=\ga$, for all $\ga$.
\eex

\bex Let $\Omega$ be an ordinal equipped with its \emph{order topology} generated by $\{0\}$ and the intervals $(\ga,\gb]$, for all  $\ga<\gb\leq\Omega$. Then $\rho$ is the function $r$ defined by
$$r(0)=0; \quad r(\ga)=\gb \text{ if $\ga=\gy+\gw^{\gb}$, for some $\gy$, $\gb$.}$$ By the Cantor normal form theorem, for any $\ga>0$, such a $\gb$ is uniquely defined.
\eex

A map $f:X\to Y$ between topological spaces is called a \emph{d-map} if $f$ is continuous, open and \emph{pointwise discrete}, that is, $f^{-1}(y)$ is a discrete subspace of $X$ for each $y\in Y$.
$d$-maps are well-known to satisfy the properties expressed in the following lemma (see \cite{BEG05}).

\bl
\begin{enumr}
\item $f^{-1}(d_Y(A))= d_X(f^{-1}(A))$, for any $A\subseteq Y$;
\item $f^{-1}:(\cP(Y),d_Y)\to (\cP(X),d_X)$ is a homomorphism of modal algebras;
\item If $f$ is onto, then $\Log(X)\subseteq \Log(Y)$.
\end{enumr}
\el

In fact, (i) is easy to check directly; (ii) follows from (i) and (iii) from (ii). From (i) we easily obtain the following corollary by transfinite induction.

\bc \label{CB-pres}
Suppose $f:X\to Y$ is a $d$-map. Then, for each ordinal $\ga$, $d_X^\ga X= f^{-1}(d_Y^\ga Y)$.
\ec

The following lemma states that the rank function, when the ordinals are equipped with their left topology, becomes a $d$-map. It is also uniquely characterized by this property.

\bl \label{rank-dmap} Let $\Omega$ be the ordinal $\rho(X)$ taken with its left topology. Then
\begin{enumr}
\item $\rho_X:X\tto \Omega$ is an onto $d$-map;
\item If $f:X\to \gl$ is a $d$-map, where $\gl$ is an ordinal with its left topology, then $f(X)=\Omega$ and $f=\rho_X$.
\end{enumr}
\el

\bp Let $\rho$ denote $\rho_X$.

(i) $\rho$ is continuous, because the set $\rho^{-1}[0,\ga)=X\setminus d^{\ga}X$ is open.

$\rho$ being open means that, for each open $U\subseteq X$, whenever $\ga \in \rho(U)$ and $\gb<\ga$ one has $\gb\in \rho(U)$. Fix an $x\in U$ such that $\rho(x)=\ga$. Consider the set  $X_\gb:=\rho^{-1}(\gb)=d^\gb X\setminus d(d^\gb X)$. For any subset $A$ of a scattered space we have $d(A)=d(A\setminus dA)$, hence $d X_\gb=d(d^\gb X)\subseteq d^\ga X$. Since $\rho(x)=\ga$ it follows that $x\in d X_\gb$. Hence $U\cap X_\gb\neq\emptyset$, that is, $\gb\in \rho(U)$.

$\rho$ being pointwise discrete means $X_\ga=\rho^{-1}(\ga)$ is discrete, for each $\ga$. In fact, $X_\ga=d^\ga X\setminus d(d^{\ga} X)$ is the set of isolated points of $d^\ga X$. Thus, it cannot help being  discrete.

(ii) Since $f$ is a $d$-map, by Corollary \ref{CB-pres} we obtain that $f^{-1}[\ga,\gl)=d^\ga X$, for each $\ga<\gl$. Hence, $f^{-1}(\ga)=\rho^{-1}(\ga)$, for each $\ga<\gl$, that is, $f=\rho$ and $f(X)=\rho(X)=\Omega$.
\ep

\bc  If $f:X\to Y$ is a $d$-map, then $\rho_X=\rho_Y\circ f$. \ec

\bp Clearly, $\rho_Y\circ f: X\to \Omega$ is a $d$-map. Statement (ii) of the previous lemma yields the result.
\ep

Note that if $U\in\tau$ is open, then the image of $U$ under the map $\rho$ is always a leftwards closed interval of ordinals and thus is itself an ordinal, which we denote $\rho(U)$.
We denote the complement of a set $\D^\alpha X$ by $O_\alpha(X)$ or simply $O_\ga$ when there is no danger of confusion.

\section{Maximal and $\ell$-maximal topologies}

First we introduce two notions: that of a \emph{rank preserving extension} of a scattered topology, and a more restrictive notion of an \emph{$\ell$-extension}. The first one is quite natural and it will help us to build a non-discrete GLP-space. The second is the one we actually need for the proof of the topological completeness theorem.

\bd Let $(X,\tau)$ be a scattered space.
\bi \item A topology $\sigma$ on $X$ is called a rank preserving \emph{extension} of $\tau$, if $\sigma\supseteq \tau$ and $\rho_{\sigma}(x)=\rho_\tau(x)$, for all $x\in X$.
\item $\sigma$ is an \emph{$\ell$-extension} of $\tau$, if it is a rank-preserving extension of $\tau$ and the identity function $id:(X,\tau)\to (X,\sigma)$ is continuous at all points of successor rank, that is,
\bi \item[($\ell$)] for any $U\in\sigma$ and any $x\in U$ with $\rho(x)\notin\Lim$ there exists $V\in\tau$ such that $x\in V\subseteq U$.\ei
\ei
\ed

We note that both notions are transitive and, in fact, define partial orders on the set of all scattered topologies on $X$. The following observation will be repeatedly used below.

\bl \label{rank-pres}
$\sigma$ is a rank preserving extension of $\tau$ iff $\rho_\tau: (X,\sigma)\tto \rho_\tau(X)$ is an open map iff $\rho_\tau(U)$ is leftwards closed, for each $U\in \sigma$.
\el
This statement follows from Lemma \ref{rank-dmap}.

We are interested in the maximal rank preserving and maximal $\ell$-extensions. These are naturally defined as follows.

\bd
\begin{enumr}
\item $(X,\tau)$ is \emph{maximal}\footnote{ In the standard terminology used in general topology, \emph{maximal} or \emph{maximal scattered} would mean something entirely different than defined here. Throughout this paper we use the term \emph{maximal} as a shorthand for \emph{maximal scattered with the given rank function}.} if $(X,\tau)$ does not have any proper rank-preserving extensions, in other words, if
$$\forall\sigma\:(\sigma\supsetneqq\tau \ \Imp \ \exists x\:\rho_{\sigma}(x)\neq \rho_{\tau}(x)).$$
\item $(X,\tau)$ is \emph{$\ell$-maximal} if $(X,\tau)$ does not have any proper $\ell$-extensions.
\end{enumr}
\ed

It is worth noting that any maximal topology is $\ell$-maximal, but not conversely.

\bl\label{l:maximal extensions}  \begin{enumr}
\item Any $(X,\tau)$ has a maximal extension;
\item Any $(X,\tau)$ has an $\ell$-maximal $\ell$-extension.
\end{enumr}
\el

\bp Consider the set of all ($\ell$-)extensions of a given topology $\tau$ ordered by inclusion. We verify, for each of the two orderings, that every chain in it has an upper bound. The result then follows by Zorn's lemma.

Suppose $(\tau_i)_{i\in I}$ is a linear chain of extensions. Then the topology $\sigma$ generated by the union $\upsilon=\bigcup_{i\in I} \tau_i$ is apparently a scattered topology containing $\tau$. Note that $\upsilon$ is closed under finite intersections and thus serves as a base for $\sigma$. Let $\rho:X\tto\Omega$ be the common rank function of each of the $\tau_i$. In order to apply Lemma \ref{rank-pres} we check that $\rho$ is open w.r.t.\ $\sigma$. In fact, any basic $U\in\upsilon$ is open in the sense of some $\tau_i$, and hence $\rho(U)$ must be open in $\Omega$. Lemma \ref{rank-pres} shows that $\rho$ is the rank function of $\sigma$. Hence (i) holds.

Suppose now that $(\tau_i)_{i\in I}$ is a chain of $\ell$-extensions. Since any $\ell$-extension is an extension, $\sigma$ (defined as above) is an extension of $\tau$. To check the condition $(\ell)$ suppose $U\in\sigma$ is given and $x\in U$ is such that $\rho(x)\notin\Lim$. Since $\sigma$ is generated by the base $\upsilon$, there exists $U'\in\upsilon$ with $x\in U'\subseteq U$. It follows that $U'\in\tau_i$ for some $i$. As $\tau_i$ is an $\ell$-extension of $\tau$, there exists $V\in\tau$ such that $x\in V\subseteq U'$. Since $U'\subseteq U$, we are done.
\ep

Next we prove a workable characterization of $\ell$-maximal topologies.

\bl\label{l:l-maximality alternatively}
Let $(X,\tau)$ be a scattered space and $\rho$ its rank function. Then $X$ is $\ell$-maximal iff the following condition holds.
\begin{enumr}
\item[$(lm)$] For any $x\in X$ with  rank $\lambda=\rho(x)\in\Lim$ and any open $V\subseteq O_\gl$, either $V\cup\{x\}\in\tau$ or there is a neighborhood $U$ of $x$ such that $\rho(V\cap U)<\gl$.
\end{enumr}
\el
Intuitively, condition $(lm)$ means that in the neighborhood of a point
$x$ of limit rank any open set $V$ is either very large (contains a
punctured neighborhood of $x$), or relatively small (there is a
punctured neighborhood whose intersection with $V$ has bounded rank).


\bp
(only if) Suppose the condition $(lm)$ is not met. Thus, there exists an $x\in X$ with $\rho(x)=\lambda\in\Lim$ and an open $V_0\subseteq O_\gl$ such that $V:=V_0\cup\{x\}$ is not open and $\rho(U\cap V_0)=\gl$, for any neighborhood $U$ of $x$.

Let us generate a new topology $\sigma$ by adding $V$ to $\tau$. We claim that $\sigma$ is an $\ell$-extension of $\tau$.
First, we observe that the neighborhood filter at any point $z\in X$, $z\neq x$, did not change. In fact, any $\sigma$-neighborhood $W$ of $z$ either contains a $\tau$-neighborhood of $z$ or contains a subset of the form $V\cap U$ where $U\in\tau$ and $z\in V\cap U${$=(V_0\cap U)\cup\{x\}$}. In the former case we are done. In the latter case, if $z\neq x$, we have {$z\in V_0\cap U\in\tau$ and $V_0\cap U\subseteq W$}.

From this observation we conclude that $id:(X,\tau)\to (X,\sigma)$ is continuous at all the points $z\neq x$, in particular, condition $(\ell)$ holds. We show that $\rho_\sigma=\rho$ by applying Lemma \ref{rank-pres}. To check that $\rho:(X,\sigma)\to \Omega$ is open it is sufficient to show that $\rho(W)$ is a neighborhood of $\gl=\rho(x)$ (in the left topology) for any $\gs$-neighborhood $W$ of $x$. For all the other points the statement is obvious by the previous observation.

We know that $W$ contains a set of the form $V\cap U$ with $x\in U\in \tau$. Clearly, $V\cap U= (V_0\cap U)\cup \{x\}$. { By the choice of $V_0$, we have} $\rho(V_0\cap U)=\gl$ and hence $\rho(W)\supseteq\rho(V\cap U) = [0,\gl]$ is a neighborhood of $\gl$, as required.

Thus, $\sigma$ is a proper $\ell$-extension of $\tau$, hence $X$ is not $\ell$-maximal.

\bigskip

{ (if) Suppose $X$ is not $\ell$-maximal and let $\sigma$ be its proper $\ell$-extension. Then the map $id:(X,\tau)\to(X,\sigma)$ is not continuous at certain points. Let $x\in X$ be such a point with the least rank $\rho(x)=\gl$. It follows from condition $(\ell)$ that $\gl\in\Lim$. Since the map $id$ is not continuous at $x$, there exists a $\sigma$-open neighborhood $V$ of $x$ which contains no $\tau$-open neighborhood of $x$. Denote $V_0:=V\cap O_\lambda$. It is clear that $V_0\in\sigma$. It follows from the minimality of $\lambda$ that $V_0\in\tau$. From the discontinuity of $id$ at $x$ we may conclude that $V_0\cup\{x\}\not\in\tau$. However, $\{x\}\cup V_0= V\cap (\{x\}\cup O_\lambda)\in\sigma$, hence, for any $\tau$-neighborhood $U$ of $x$ we have $(U\cap V_0)\cup\{x\}=U\cap(V_0\cup\{x\})$ is a $\sigma$-neighborhood of $x$. It follows that $\rho(U\cap V_0)=\gl$. Thus $x$ and $V_0$ witness that the condition $(lm)$ is violated for $\tau$. } \ep

Our next objective is to show that whenever $f:X\to Y$ is an onto $d$-map and $Y'$ is any $\ell$-maximal $\ell$-extension of $Y$, one can always find a suitable $\ell$-maximal $\ell$-extension $X'$ of $X$ so that $f:X'\to Y'$ is still a $d$-map. We need an auxiliary lemma.

\bl\label{l:lifting d-maps along l-extensions} Let $f:X\to Y$ be a $d$-map between a scattered space $X=(X,\tau)$ and an $\ell$-maximal space $Y=(Y,\sigma)$. Let $X'=(X,\tau')$ be any $\ell$-extension of $X$. Then $f:X'\to Y$ is also a $d$-map.
\el
\bp
That $f:X'\to Y$ is continuous and pointwise discrete follows from the fact that $\tau'\supseteq\tau$. We only have to show that $f:X'\to Y$ is open. For the sake of contradiction suppose $f$ is not. Then there exists a point $x\in X'$ and a neighborhood $U\in\tau'$ of $x$ such that $f(U)$ does not contain a neighborhood of $y=f(x)$. We can take such an $x$ of the minimal possible rank $\gl$. This ensures that the restriction of $f$ to the subspace $O_\gl(X')$ is open, hence a $d$-map. (Since $X'$ is a rank preserving  extension of $X$, the set $O_\gl=O_\gl(X)$ is the same as $O_\gl(X')$.)

Since $id:X\to X'$ is continuous at the points of non-limit ranks and $f:X\to Y$ is a $d$-map, we observe that $\gl\in\Lim$. Otherwise, for a sufficiently small $\tau$-neighborhood $V$ of $x$ we would have $V\subseteq U$, and then $f(V)\subseteq f(U)$ would be a $\sigma$-neighborhood of $f(x)$.

Since $O_\gl\in\tau$, we may assume that the selected neighborhood $U$ has the form $U=U_0\cup\{x\}$ where ${U}_0\subseteq O_\gl$ and $U_0\in\tau$. Thus, $\rho(x)=\gl\in\Lim$, $V_0:=f(U_0)$ is open, and $V:=f(U)= V_0\cup\{y\}$ is not open in $Y$. Since $Y$ is $\ell$-maximal, by Lemma \ref{l:l-maximality alternatively} we obtain an open neighborhood $W$ of $y$ such that $\gb:=\rho(V_0\cap W)<\gl$. We notice that $f(U_0\cap f^{-1}(W))= V_0\cap W$. Hence, $\rho(U_0\cap f^{-1}(W))=\gb$. Since $f^{-1}(W)\in\tau$ and $U\in\tau'$ we obtain that $U_1:=U\cap f^{-1}(W)=(U_0\cap f^{-1}(W))\cup\{x\}$ is a $\tau'$-open neighborhood of $x$. Therefore, { on the one hand,} $\rho(U_1)=\rho_{\tau'}(U_1)=[0,\gl]$, as $\tau'$ is a rank preserving extension of $\tau$. However, { on the other hand,} $\rho(U_1) = \rho((U_0\cap f^{-1}(W))\cup\{x\})=\gb \cup\{\gl\}$, a contradiction.
\ep

\bl \label{pullback}
Let $X=(X,\tau)$ and $Y=(Y,\sigma)$ be scattered spaces, let $Y'=(Y,\sigma')$ be an $\ell$-maximal $\ell$-extension of $Y$ and let $f:X\to Y$ be a $d$-map. Then there exists an $\ell$-maximal $\ell$-extension $X'=(X,\tau')$ of $X$ such that $f:X'\to Y'$ is a  $d$-map.
$$
\xymatrix{X \ar[r]_d \ar@{.>}[d]_{lm} & Y \ar[d]^{lm} \\ X' \ar@{.>}[r]_d & Y'
}
$$
\el

\bp
It is easily seen that the collection $\theta=\{ f^{-1}(U) : U\in\sigma'\}$ qualifies for a topology on $X$. Since $f:X\to Y$ is continuous, $\theta$ contains $\tau$. It is readily seen that $f:(X,\theta)\to (Y,\sigma')$ is a $d$-map. Thus $\theta$ is a rank preserving extension of $\tau$.

To see that the condition $(\ell)$ is met, take any $x\in X$ of successor rank and any $f^{-1}(V)\ni x$ such that $V\in\sigma'$. Since $f(x)$ is of the same rank as $x$, by condition $(\ell)$ applied to $\sigma'$, there exists $U\in\sigma$ with $f(x)\in U\subseteq V$. It follows that $x\in f^{-1}(U)\subseteq f^{-1}(V)$ and $f^{-1}(U)\in\tau$.

Therefore, $\theta$ is an $\ell$-extension of $\tau$. Take any $\ell$-maximal $\ell$-extension $\tau'$ of $\theta$. By Lemma~\ref{l:lifting d-maps along l-extensions} we obtain that $f:(X',\tau')\to Y'$ is an onto $d$-map. Since $\tau'$ is also an $\ell$-maximal $\ell$-extension of $\tau$, the proof is finished.
\ep

\section{Building a non-discrete GLP-space}

Recall that the \emph{next topology} $\tau^+$ on $X$ is generated by $\tau$ and $\{d(A) : A\subseteq X\}$.
Let $X^+$ denote the space $(X,\tau^+)$. The following lemma gives a useful characterization of the next topology for $\ell$-maximal spaces.

\bl \label{lmax-plus} Suppose $(X,\tau)$ is $\ell$-maximal. Then $\tau^+$ is generated by $\tau$ and the sets $\{d^{\gb+1}(X):\gb<\rho(X)\}$.
\el

\bp
Let $(X,\tau)$ be $\ell$-maximal and let $\tau'$ denote the topology generated by $\tau$ and the sets $\{d^{\gb+1}(X):\gb<\rho(X)\}$. It is clear that each set $d^{\gb+1}(X)=d(d^{\gb}X)$ is open in $\tau^+$. We show the converse.

Let $A\subseteq X$, we show that $d(A)$ is open in $\tau'$. Consider any $x\in d(A)$ and let $\ga=\rho(x)$. If $\ga$ is not a limit ordinal, $\{x\}$ is open in $\tau'$. In fact, since $\rho$ is a $d$-map, $\rho^{-1}(\ga)$ is discrete as a subspace of $(X,\tau)$. Moreover, $\rho^{-1}(\ga)=d^{\ga}(X)\setminus d^{\ga+1}(X)$, hence it is clopen in $\tau'$. It follows that $x$ is isolated in $\tau'$.

Suppose $\ga\in\Lim$ and let $C$ denote the interior of $O_\ga\setminus A$. Since $x\in dA$ we have $\{x\}\cup C\notin\tau$. Hence, by condition $(lm)$, there is an open $U\in\tau$ with $x\in U$ and a $\gb<\ga$ such that $U\cap C\subseteq O_\gb$. Consider $V:= U\cap d^{\gb+1}X$. Since $U$ is open in $\tau$, $V$ is open in $\tau'$. Moreover, $x\in V$. Thus, we only have to show that $V\subseteq dA$.

Suppose the contrary that $z\in V\setminus dA$ for some $z$. Then there exists an open set $U_z\cup \{z\}$ such that $U_z\cap A =\emptyset$ and $U_z\subseteq O_\ga$. It follows that $U_z\subseteq C$ and hence $U_z\cap U\subseteq O_\gb$. Since $z\in V\subseteq U$, we have that $U':=(U_z\cap U)\cup \{z\}=(U_z\cup\{z\})\cap U$ is an open neighborhood of $z$. As $\rho$ is an open map, $\rho(U')$ must be leftwards closed. We have  $\rho(z)\geq \gb$, since $z\in d^{\gb+1}X$, however $\rho(U_z\cap U)\subseteq \rho(O_\gb)\subseteq \gb$, a contradiction. \ep

\bl \label{lift+} Suppose $(X,\tau)$ is $\ell$-maximal and $f:X\to Y$ a $d$-map.
Then $f$ is a $d$-map between $X^+$ and $Y^+$.
\el

\bp We only have to show that $f:X^+\to Y^+$ is open. From the previous lemma we know that $\tau^+$ is generated by $\tau$ and $d_X^{\gb+1}X$ for $\gb<\ga$. Consider {a $\tau^+$-open} set of the form $A\cap d_X^{\gb+1}X$. Since $f^{-1}(d_Y^{\gb+1}Y)= d_X^{\gb+1}X$ ($f$ is rank preserving), we have $f(A\cap d_X^{\gb+1}X)= f(A)\cap d_Y^{\gb+1}Y$, which is open {in $Y^+$}.
\ep

\brem In general, the `next topology' operation is non-monotonic: There is a space $X$ such that $X^+$ is discrete while $(X')^+$ is not, where $X'$ is some maximal extension of $X$.
\erem

Let $\Omega$ denote an ordinal with its left topology. It is easy to check (see \cite{BBI09}) that $\Omega^+$ coincides with the usual order topology on $\Omega$. Let $r$ denote its rank function (see above). In general, for an arbitrary scattered space $X$ let $\rho_X^+$ denote the rank function of $X^+$.

\bc \label{lmax-rank+} If $X$ is $\ell$-maximal, then $\rho_X^+=r\circ\rho_X$. \ec

\bp Let $\Omega:=\rho(X)$ be the rank of $X$. Consider the $d$-map $\rho:X\tto\Omega$.
By Lemma \ref{lift+}, $\rho:X^+\tto \Omega^+$ is a $d$-map.
Since $r$ is the rank function of $\Omega^+$, $r:\Omega^+\to \Omega$ is also a $d$-map.
Hence, $r\circ\rho:X^+\to \Omega$ is a $d$-map and coincides with the rank function of $X^+$.
\ep

\brem
For an arbitrary scattered space $X$ we only have $\rho_X^+\leq r\circ\rho_X$.
\erem

Now we are ready to specify a suitable class of GLP-spaces which will be used for the topological completeness proof.
\bd
Let $(X,\tau)$ be a scattered space. A poly-topological space $(X,\tau_0,\tau_1,\dots)$ is called an \emph{lme-space based on $\tau$} if\footnote{The abbreviation \emph{lme} stands for \emph{limit maximal extension}.}
$\tau_0$ is an $\ell$-maximal $\ell$-extension of $\tau$ and,
for each $n$, $\tau_{n+1}$ is an $\ell$-maximal $\ell$-extension of $\tau_n^+$.
\ed

Clearly, any lme-space is a GLP-space. $(X,\tau_0,\tau_1,\dots)$ is called an \emph{ordinal lme-space} if $X$ is an ordinal (or an interval of the ordinals) and $\tau$ is the order topology on $X$.
Given an lme-space $X$, let $\rho_n$ denote the rank function of $\tau_n$.

\bl $\rho_{n+1}=r\circ \rho_n$.
\el
\bp $\tau_{n+1}$ has the same rank function as $\tau_n^+$, being its $\ell$-extension, hence $\rho_{n+1}=\rho_n^+$. By Corollary \ref{lmax-rank+}, $\rho_n^+=r\circ \rho_n$.\ep

Now we can give an example of a GLP-space in which all topologies are non-discrete.
Take any scattered space $(X,\tau)$ whose rank $\Omega$ satisfies $\gw^\Omega=\Omega$, for example, $X=\ge_0$ with the order topology. Generate some lme-space $(X,\tau_0,\tau_1,\dots)$ based on $\tau$. Then clearly
$\rho_n(X)=r^n(\rho_0(X))=r^n(\Omega)=\Omega$, for each $n$. In particular, any topology $\tau_n$ is non-discrete. Thus, we have proved

\bt \label{lme-spaces}
There is a countable GLP-space $(X,\tau_0,\tau_1,\dots)$ such that each $\tau_n$ is non-discrete.
\et

\section{Topological completeness of GLP}

In this section we reduce the construction of a poly-topological space whose logic is $\ml{GLP}$ to a technical lemma. The rest of the paper is devoted to a proof of this lemma.

Our proof of topological completeness will make use of a subsystem of $\GLP$ introduced in \cite{Bek10} and  denoted \textbf{J}. This logic is defined by weakening axiom (iv) of $\GLP$ to the following axioms (vi) and (vii) both of which are theorems of $\GLP$: \begin{itemize}
\item[(vi)] $[m]\phi\rightarrow[n][m]\phi$, for $n\geq m$;
\item[(vii)] $[m]\phi\rightarrow[m][n]\phi$, for $n>m$.
\end{itemize}
\textbf{J} is the logic of a simple class of frames, which is established by standard methods \cite[Theorem 1]{Bek10}.

\begin{lem} \label{jlem} \textnormal{\textbf{J}} is sound and complete with respect to the class of (finite) frames $(W,R_0,R_1, \dots)$ such that, for all $x,y,z \in W$, \begin{enumerate}
\item $R_k$ are transitive and dually well-founded binary relations;
\item If $xR_ny$, then $xR_m z$ iff $yR_m z$, for $m<n$;
\item $xR_my$ and $yR_n z$ imply $xR_m z$, for $m<n$.
\end{enumerate} \end{lem}

Let $R^*_n$ denote the transitive closure of $R_n\cup R_{n+1}\cup \dots$, and let $E_n$ denote the reflexive, symmetric, transitive closure of $R^*_n$. Obviously, each $E_{n+1}$ refines $E_n$. We call each $E_n$ equivalence class a \textit{$n$-sheet}. By {\it 2}., all points in an $n$-sheet are $R_m$ incomparable, for $m<n$. But $R_n$ defines a natural ordering on $n+1$-sheets in the following sense: if $\alpha$ and $\beta$ are $n+1$-sheets, then $\alpha R_n \beta$, iff $\exs{x\in\alpha} \exs{y\in\beta} xR_n y$. By the standard techniques, one can improve on Lemma \ref{jlem} to show that \textbf{J} is complete for such frames, in which the set of $n+1$-sheets contained in each $n$-sheet is a tree under $R_n$, and if $\alpha R_n \beta$ then $xR_n y$ for all $x \in \alpha$, $y \in \beta$ (see \cite[Theorem 2 and Corollary 3.3]{Bek10}). Every such structure is automatically a J-frame, we call such frames \emph{tree-like J-frames}.

As shown in \cite{Bek10}, \textbf{GLP} is reducible to \textbf{J} in the following sense. Let \[M(\phi) := \bigwedge_{i<s}\bigwedge_{k=m_i+1}^n ([m_i]\phi_i\rightarrow[k]\phi_i),\] where $[m_i]\phi_i$, $i<s$, are all subformulas of $\phi$ of the form $[m]\psi$ and $n:=\max_{i<s}m_i$. Also, let $M^+(\phi):=M(\phi)\land\bigwedge_{m\leq n}[m]M(\phi)$.\footnote{The formula $M(\phi)$ was defined in \cite{Bek10} incorrectly, however with the present modification everything in \cite{Bek10} works.}

\begin{prop} [\cite{Bek10}] $\GLP\vdash\phi$ iff $\; \textnormal{\textbf{J}}\vdash M^+(\phi)\rightarrow\phi$. \end{prop}

For the proof below we will only need the trivial implication from the right to the left. We obtain another proof of this proposition as a byproduct of the topological completeness proof below.

\bigskip
Let $\mathcal{L}_n$ denote the modal language with modalities $[0],[1],\dots,[n]$. Denote by $\ml{J}_{n}$ the logic $\ml{J}$ restricted to $\mathcal{L}_n$. Analogously for $\ml{GLP}_n$. Let $T=(T,R_0,\dots, R_n)$ be a tree-like $J_n$-frame (or $J_n$-tree for short). Recall that $w\in T$ is called a \emph{hereditary $k$-root} if for no $j\geq k$ and no $v\in T$ is it true that $v R_j w$. Note that since $T$ is a $J_n$-tree, for each $w\in T$ and each $k\leq n$ there exists a hereditary $k$-root $v\in T$ such that $v=w$ or  $vR_k w$.

\bd
We view $T$ as a poly-topological space $T=(T,\sigma_0,\dots, \sigma_n)$ by considering all $R_i$-upsets to be $\sigma_i$-open. Given a GLP$_n$ space $X=(X,\tau_0,\dots,\tau_n)$ and a map $f:X\to T$ we will say that $f$ is a \emph{$J_n$-morphism} iff:

\begin{itemize}
\item[($j_1$)] $f:(X,\tau_n)\to (T,\sigma_n)$ is a $d$-map;
\item[($j_2$)] $f:(X,\tau_k)\to (T,\sigma_k)$ is an open map for all $k\leq n$;
\item[($j_3$)] For each $k<n$ and each hereditary $(k+1)$-root $w\in T$, the sets $f^{-1}(R_k^*(w))$ and $f^{-1}(R_k^*(w)\cup\{w\})$ are open in $\tau_k$;
\item[($j_4$)] For each $k<n$ and each hereditary $(k+1)$-root $w\in T$, the set $f^{-1}(w)$ is a $\tau_k$-discrete subspace of $X$.
\end{itemize}

Here $R_k^*(w)$ denotes the set $\bigcup\limits_{i=k}^n R_i(w)$. Also notice that $(j_1)$ would follow from $(j_2)$--$(j_4)$ if one also stated
them for $k=n$ assuming that $R_{n+1}=\emptyset$. In this case each element of $T$ would be an $(n+1)$-root. The same definition also applies to general J$_n$-models.
\ed

A $J_n$-morphism $f:X\to T$ can be thought of as a map which is a weak kind of $d$-map from $(X,\tau_k)$ to $(T,\gs_k)$, for each $k\leq n$. As a consequence, we obtain the following simple but useful observation.

\bl
Suppose $X,Y$ are GLP$_n$-spaces, $g:(Y,\theta_k)\to (X,\tau_k)$ is a $d$-map, for each $k\leq n$, and $f:X\to T$ is a $J_n$-morphism. Then $f\circ g$ is a $J_n$-morphism from $Y$ to $T$.
\el

Let $\tilde\D(A)$ abbreviate $X\setminus\D(X\setminus A)$. Obviously, $x\in\tilde\D(A)$ iff $A$ contains some punctured neighborhood of $x$.
\begin{lem} \label{l:j34}
Conditions ($j_3$) and ($j_4$) together are equivalent to the following one: for any hereditary $(k+1)$-root $w$, \[ f^{-1}(R_k^*(w)\cup\{w\}) \subseteq \tilde\D_k(f^{-1}(R_k^*(w))). \leqno (*) \]
\end{lem}

\begin{proof} Suppose $(*)$ holds. Then $f^{-1}(R_k^*(w))$ contains a punctured neighborhood of every point $a\in f^{-1}(R_k^*(w)\cup\{w\})$, hence a neighborhood of every $a\in f^{-1}(R_k^*(w))$. So, $f^{-1}(R_k^*(w))$ is open. It also follows that $f^{-1}(R_k^*(w)\cup\{w\})$ contains a neighborhood of every point $a\in f^{-1}(R_k^*(w)\cup\{w\})$, hence $f^{-1}(R_k^*(w)\cup\{w\})$ is also open.

To show that $f^{-1}(w)$ is discrete assume $a\in f^{-1}(w)$. Select a punctured neighborhood $V_a$ of $a$ such that $V_a\subseteq f^{-1}(R_k^*(w))$. Since $w\notin R_k^*(w)$ we have $V_a\cap f^{-1}(w)=\emptyset$, as required.

Suppose $(j_3)$ and $(j_4)$ hold, we show $(*)$. Assume $a\in f^{-1}(R_k^*(w)\cup\{w\})$. We have to construct a punctured neighborhood of $a$ contained in $f^{-1}(R_k^*(w))$. Consider  $$U:=f^{-1}(R_k^*(w)\cup\{w\})=f^{-1}(R_k^*(w))\cup f^{-1}(\{w\}).$$ By the first part of $(j_3)$, $U$ is a neighborhood of $a$. If $a\in f^{-1}(R_k^*(w))$ then $V:=f^{-1}(R_k^*(w))$ is a neighborhood of $a$ by the second part of $(j_3)$, so $V-\{a\}$ is as required. If $a\in f^{-1}(w)$ then by $(j_4)$ there is a neighborhood $V_a$ such that $V_a\cap f^{-1}(w)=\{a\}$. Then, $$V_a\cap U= (V_a\cap f^{-1}(R_k^*(w)))\cup \{a\}$$ is a neighborhood of $a$. Then, $(V_a\cap U)\setminus\{a\}$ is a punctured neighborhood of $a$ contained in $f^{-1}(R_k^*(w))$.
\end{proof}

The following theorem is crucial.

\begin{thm}\label{t:main-1}
Let $X$ be a ${GLP}_n$-space, $T$ a ${J}_n$-tree, $f:X\to T$ a ${J}_n$-morphism and $\varphi$ a $\mathcal{L}_n$-formula. Then $X\models\varphi$ iff $T\models M^+(\varphi)\to\varphi$.
\end{thm}
\begin{proof}
Suppose $T\nmodels M^+(\varphi)\to\varphi$. Then for some valuation $\nu$ on $T$ and some point $w\in T$ (assume without loss of generality that $w$ is the hereditary $0$-root of $T$) we have that $w\in\nu(M^+(\varphi))$ but $w\not\in\nu(\varphi)$. Consider a valuation $\nu'$ on $X$ by taking $\nu'(p)=f^{-1}(\nu(p))$.

\begin{lem}
For all subformulas $\theta$ of $\varphi$, we have $\nu'(\theta)=f^{-1}(\nu(\theta))$.
\end{lem}

\begin{proof}
We argue by induction on the complexity of $\theta$.
If $\theta$ is a propositional letter, the claim is provided by the definition of $\nu'$.
The case of propositional connectives is trivial.

If $\theta=[n]\psi$, then the claim follows by condition $(j_1)$ of $f$ being a $J_n$-morphism.

Suppose $\theta=[k]\psi$ for some $k<n$. To show that $\nu'(\theta)\subseteq f^{-1}(\nu(\theta))$ assume $x\in\nu'(\theta)$. Then there exists a $U\subseteq X$ such that $\{x\}\cup U\in\tau_k$ and $U\subseteq\nu'(\psi)$. By IH we obtain $U\subseteq f^{-1}(\nu(\psi))$. Hence $f(U)\subseteq f(f^{-1}(\nu(\psi)))=\nu(\psi)$. By $(j_2)$, the set $f(\{x\}\cup U)=\{f(x)\}\cup f(U)$ is an $R_k$-upset and so $R_k(f(x))\subseteq f(U)\subseteq\nu(\psi)$. It follows that $f(x)\in\nu([k]\psi)$. In other words, $x\in f^{-1}(\nu(\theta))$.

For the converse inclusion suppose $x\in f^{-1}(\nu(\theta))$, that is,  $f(x)\models [k]\psi$. We must show $x\in \nu'(\theta)$.
By the induction hypothesis,
$$\nu'(\theta)=\tilde\D_k(\nu'(\psi))=\tilde\D_k(f^{-1}(\nu(\psi))).$$

Let $v\in T$ be a hereditary $(k+1)$-root such that $v=f(x)$ or $v R_{k+1} f(x)$. Since $v$ and $f(x)$ are in the same $(k+1)$-sheet, $R_k(v)=R_k(f(x))$. Thus $v\models[k]\psi$. We also have $v\models M^+(\varphi)$. In particular, $v\models [k]\psi\to[k']\psi$ for any $k'$ with $k\leq k'\leq n$ and hence $v\models[k']\psi$. It follows that for each $k'$ between $k$ and $n$ we have $R_{k'}(v)\subseteq\nu(\psi)$. Therefore $R^*_{k}(v)\subseteq \nu(\psi)$ and hence $f^{-1}(R^*_{k}(v))\subseteq f^{-1}(\nu(\psi))$.
By the construction of $v$, $x\in f^{-1}(R^*_{k}(v)\cup\{v\})$. Hence,  by Lemma \ref{l:j34},
$$x\in \tilde\D_k(f^{-1}(R^*_{k}(v)))\subseteq\tilde\D_k(f^{-1}(\nu(\psi))),$$
as required. \end{proof}

From this lemma we obtain $y\notin \nu'(\varphi)=f^{-1}(\nu(\varphi))$, for any $y$ with $f(y)=w$. Consequently, $X\nmodels\varphi$.
\end{proof}

The proof of the following lemma will be provided later on.

\begin{lem}[main] \label{l:key construction}
For each finite ${J}_n$-tree $T$ there exist an ordinal lme-space $X=([1,\gl],\tau_0,\dots,\tau_n)$ and an onto ${J}_n$-morphism $f:X\tto T$, where $\gl<\epsilon_0$.
\end{lem}

Using this lemma we can prove that the logic $\ml{GLP}$ is topologically complete. Let $\mathcal{L}_\omega$ denote the modal language with modalities $[k],\ k<\omega$.

\begin{thm}\label{t:main-2}
Let $\varphi$ be a formula of $\mathcal{L}_\omega$. If $\GLP\nvdash\varphi$ then $\varphi$ can be refuted on a GLP-space.
\end{thm}
\begin{proof} Suppose $\GLP\nvdash\varphi$ and let $n$ be the maximal such that $[n]$ occurs in $\varphi$. Obviously,  $\J_n\nvdash M^+(\varphi)\to\varphi$. Then there exists a finite $J_n$-tree $T$ such that $T\nmodels M^+(\varphi)\to\varphi$. By Lemma~\ref{l:key construction} there exists a {GLP}$_n$-space $X=([1,\gl],\tau_0,\dots,\tau_n)$ and a $J_n$-morphism $f:X\tto T$. By Theorem~\ref{t:main-1} we have $X\nmodels\varphi$. Let $X_\omega$ denote the GLP-space $X_\omega=(X,\tau_0,\dots,\tau_n,\tau_{n+1},\dots)$ where topology $\tau_{i}$ is discrete for $i>n$. It is obvious that $X_\omega\nmodels\varphi$.
\end{proof}

The topological completeness theorem can also be stated in a stronger \emph{uniform} way. Recall that $\ge_0$ is the supremum of the countable ordinals $\gw_k$ recursively defined by $\gw_0=1$ and $\gw_{k+1}=\gw^{\gw_k}$.

\bt
There is an ordinal lme-space $X=(\ge_0,\tau_0,\tau_1,\dots)$ such that $\Log(X)=\GLP$.
\et

\bp Let $\phi_0$, $\phi_1$, \dots be an enumeration of all the formulas of $\cL_\gw$. Using Theorem \ref{t:main-2} select ordinal lme-spaces  $X_i=([1,\gl_i],\tau_0^i,\tau^i_1,\dots)$ in such a way that $X_i\nmodels \phi_i$, for each $i<\gw$. We can assume that $\gl_i<\ge_0$, for each $i<\gw$. Consider the ordinal $\gl:=\sum_{i<\gw}\gl_i$. The interval $[1,\gl)$ is naturally identified with the disjoint union $\bigsqcup_{i<\gw}[1,\gl_i]$. Hence, we can define the topologies $\tau_i$ on $[1,\gl)$ in such a way that  $X=([1,\gl),\tau_0,\tau_1,\dots)$ is isomorphic to the topological sum $\bigsqcup_{i<\gw} X_i$. Then clearly $\gl\leq\ge_0$ and each formula $\phi$ such that $\GLP\nvdash\phi$ is refutable on $X$. Hence, $\Log(X)=\GLP$.

In fact, $\gl$ must coincide with $\ge_0$. Assume $\gl<\gw_n$. Then for the topology $\tau_n$ we have $\rho_n(X)\leq r^{n+1}(\gw_n)=0$ by Theorem \ref{lme-spaces}. However, this contradicts the fact that the unprovable formula $[n]\bot$ is refutable in $X$. Therefore, $\gl=\ge_0$ and $X$ is isomorphic to an ordinal lme-space based on $\ge_0$.
\ep

In order to prove the main lemma we introduce the notion of \emph{d-product} of scattered spaces.

\section{$d$-product}

\bd
Let $(X,\tau_X)$ and $(Y,\tau_Y)$ be any topological spaces. We define their \emph{$d$-product space}  $(Z,\tau_Z)$, denoted $X\otimes_d Y$, as follows.

Notice that $Y$ is a union of its isolated points and limit points, $Y=\iso(Y)\cup d(Y)$. For all $y\in \iso(Y)$, let $X_y$ denote pairwise disjoint copies of $X$, and let  $i_y:X\to X_y$ be the associated homeomorphism maps.

Let $Z_0$ be the topological sum of $\{X_y:y\in \iso(Y)\}$, that is,
$Z_0:=\bigsqcup_{y\in \iso(Y)} X_y$. $Z_0$ can also be defined as the cartesian product $X\times \iso(Y)$ of $X$ and the discrete space $\iso(Y)$. Projection $\pi_0:Z_0\tto X$ is defined in a natural way, that is, $\pi_0(i_y(x))=x$, for each $y\in \iso(Y)$.

Let $Z_1$ be a copy of the set $dY$ disjoint from $Z_0$, and $\pi:Z_1\to dY$ the associated bijection. Put  $Z:=Z_0\cup Z_1$. We set $\pi_1(x):=y$, if $x\in X_y$ and $y\in \iso(Y)$, and $\pi_1(x):=\pi(x)$, if $x\in Z_1$. It is also convenient to let $X_y:=\{y\}$, if $y\in dY$, thus, $X_y=\pi_1^{-1}(y)$, for each $y\in Y$.

Let a topology $\tau_Z$ on $Z$ be generated by the one inherited from $Z_0$ (with the basic open sets $\{i_y(V):V\in\tau_X,\ y\in \iso(Y)\}$) and by all sets $\{\pi_1^{-1}(U): U\in \tau_Y\}$.
\ed

We note that, for each $y\in \iso(Y)$ and $U\subseteq Y$, the set $\pi_1^{-1}(U)\cap X_y$ is either empty or coincides with $X_y$. Hence, the above basic open sets form a base of topology $\tau_Z$. It follows that any open set of $\tau_Z$ has the form $V\cup \pi_1^{-1}(U)$, where $V$ is open in $Z_0$ and $U\in\tau_Y$. (Pay attention that this union need not be disjoint.) It also follows that the topologies induced from $Z$ on $Z_0$ and $Z_1$ are homeomorphic to those of the product $X\times \iso(Y)$ and $Y$, respectively.

\bigskip
As a typical example, consider the $d$-product of two compact ordinal spaces $[1,\lambda]$ and $[1,\mu]$ taken with their interval topologies. We claim that $[1,\lambda]\otimes_d [1,\mu]$ is isomorphic to $[1,\lambda\mu]$ (with the interval topology). Indeed, every $\ga\in [1,\lambda\mu]$ either has the form $\lambda\gb$ with $\gb\in\Lim$, or belongs to a (clopen) interval $I_{\gb+1}:=[\lambda\gb+1, \lambda(\gb+1)]$ isomorphic to $[1,\lambda]$. In the former case, $\ga=\gl\gb$ corresponds to a limit point $\gb\in [1,\mu]$. In the latter case, $\ga$ belongs to a copy of $[1,\lambda]$ corresponding to an isolated point $\gb+1$ of $[1,\mu]$.

The described bijection is, in fact, a homeomorphism: an interval of the form $(\gd,\ga]$, where $\gd<\ga\leq \lambda\mu$ is a neighborhood of $\ga$ in the $d$-product topology. This is clear if $\ga\in I_{\gb+1}$. If $\ga=\gl\gamma$ with $\gamma\in\Lim$, then for all sufficiently large $\gb<\gamma$, $I_{\gb}\subseteq (\gd,\ga]$, if $\gb\in\Suc$, and $\gl\gb\in (\gd,\ga]$, if $\gb\in\Lim$; hence, the claim. The converse is also clear: a neighborhood of $\ga$ in the $d$-product topology contains a suitable interval of the form $(\gd,\ga]$.

\bl
\begin{enumr}
\item $\pi_0:Z_0\tto X$ is a $d$-map;
\item The map $\pi_1:Z\tto Y$ is continuous and open.
\end{enumr}
\el

\bp (i) This follows from the fact that $Z_0$ is homeomorphic to the product $X\times\iso(Y)$ with $\iso(Y)$ discrete.

(ii) The continuity of $\pi_1$ is clear. To show that it is open, we check that $\pi_1(U)$ is open in $Y$, for each basic open set $U$ of $Z$. If $U$ is $\pi_1^{-1}(V)$ for a set $V\in \tau_Y$, we are done. If $U=i_y(V)$, for some nonempty $V\in\tau_X$ and $y\in\iso(Y)$, then $\pi_1(U)=\{y\}\in\tau_Y$ as well.
\ep

The following observations will also be helpful.
\bl \label{dop}
\begin{enumr}
\item \label{punc} Suppose $x\in Z_1$. Then $U$ is a punctured neighborhood of $x$ in $\tau_Z$ iff $\{y\in Y: X_y\subseteq U\}$ is a punctured neighborhood of $\pi_1(x)$ in $\tau_Y$.
\item Let $A\subseteq Z$, $x\in Z_1$. Then, $x\in d_Z(A)$ iff $\pi_1(x)\in d_Y\{y\in Y : A\cap X_y\neq \emptyset\}$.
\end{enumr}
\el

Clearly, $X\otimes_d Y$ is scattered if so are $X$ and $Y$. Let us compute the rank function of $X\otimes_d Y$.

\bl \label{rankl}
\begin{enumr}
\item
If $x\in Z_0$ then $\rho_Z(x)=\rho_X(\pi_0(x))$.
\item
If $x\in Z_1$ then $\rho_Z(x)= \rho(X) + \rho_{dY}(\pi_1(x))$. (Obviously, $1+\rho_{dY}(y)=\rho_Y(y)$.)
\end{enumr}
\el

\bp For (i), we just notice that $\rho_Z(x)=\rho_{Z_0}(x)$, since $Z_0$ is open in $Z$. Since $\pi_0:Z_0\to X$ is a $d$-map, we have $\rho_{Z_0}(x)=\rho_X(\pi_0(x))$.

For (ii) we first prove that $Z_1\subseteq d_Z^\gb(Z)$, for each $\gb<\rho(X)$. This goes by transfinite induction on $\gb$. The cases when $\gb=0$ or $\gb\in\Lim$ are easy. Suppose the claim is true for all $\ga\leq\gb$. We prove that $Z_1\subseteq d_Z^{\gb+1}(Z)=d_Z(d_Z^\gb(Z))$. By (i), if $\gb<\rho(X)$ then $d_Z^\gb(X_y)=i_y(d_X^\gb(X))\neq \emptyset$, for all $y\in \iso(Y)$. Hence, any $y\in Z_1$ is a limit point of $d^\gb(Z_0)$, hence of $d_Z^\gb(Z)$, as required.

As a consequence we obtain that $d_Z^{\rho(X)}(Z)=Z_1$. Hence, $d_Z^{\rho(X)+\ga}(Z)=d_Z^{\ga}(Z_1)=\pi_1^{-1}(d_Y^{1+\ga}(Y))$, for each $\ga$. \ep

Next we would like to show that $d$-product is well-behaved w.r.t.\ $\ell$-extensions.

\bl \label{prod-lex} Suppose $X'$, $Y'$ are $\ell$-extensions of $X$, $Y$, respectively.
Then $X'\otimes_d Y'$ is an $\ell$-extension of $X\otimes_d Y$.
\el

\bp The rank function is preserved by the previous lemma. We only have to check that the identity function $id:X\otimes_d Y \to X'\otimes_d Y'$ is continuous at the points $x$ of successor rank. Let $Z=X\otimes_d Y$. If $x\in Z_0$, the claim follows from the hypothesis about $X'$. Suppose $x\in Z_1$. By Lemma \ref{rankl} $\rho_Y(\pi_1(x))$ is not a limit. Consider a basic open neighborhood $V'$ of $x$ in $Z'=X'\otimes_d Y'$. $V'$ has the form $\pi_1^{-1}(U')$, where $U'$ is a $Y'$-neighborhood of $\pi_1(x)$. Since $Y'$ is an $\ell$-extension of $Y$, there is a $Y$-neighborhood $U\subseteq U'$ such that $\pi_1(x)\in U$.
Then $x\in \pi_1^{-1}(U)\subseteq V'$, as required.
\ep

\bl \label{prod-lexm} Suppose $X$ and $Y$ are $\ell$-maximal and $\rho(X)\in\Suc$. Then $X\otimes_d Y$ is $\ell$-maximal.  \el

\bp We use Lemma \ref{l:l-maximality alternatively}. Let $Z=X\otimes_d Y$ and suppose $x\in Z$ and $\rho_Z(x)=\gl\in\Lim$. Consider any open $V\subseteq O_\gl(Z)=\{z\in Z:\rho_Z(z)<\gl\}$. We show that either $V\cup\{x\}$ is open, or there is a open neighborhood $U_x$ of $x$ such that $\rho_Z(V\cap U_x)<\gl$.

Case 1: $x\in Z_0$. In this case, $V\subseteq \cO_\gl(Z)\subseteq Z_0$ by Lemma \ref{rankl} (i). Also, $Z_0$ is $\ell$-maximal as a topological sum of $\ell$-maximal spaces. Hence, the claim follows from $\ell$-maximality of $Z_0$.

Case 2: $x\in Z_1$. In this case we represent $V$ as a union $W\cup \pi_1^{-1}(U)$, where $W$ is open in $Z_0$ and $U$ in $Y$.  Let $y:=\pi_1(x)$ and let $\mu:=\rho_Y(y)$. By Lemma \ref{rankl} (ii) we have $\rho(X)+\mu'=\gl$ where $\mu=1+\mu'$. Since $\gl$ is a limit ordinal, so is $\mu$ (unless $\mu'=0$ and $\gl=\rho(X)$, but then $\rho(X)$ would be a limit). Hence, we can use the $\ell$-maximality of $Y$ for $y$, $\mu$, and $U\subseteq O_{\mu}(Y)$.

Suppose $\rho_Y(U\cap U_y)=\gb<\mu$, for some open neighborhood $U_y$ of $y$ in $Y$. Let $U_x:=\pi_1^{-1}(U_y)$. Then $U_x\cap \pi_1^{-1}(U)=\pi_1^{-1}(U\cap U_y)$ is a neighborhood of $x$ (by the continuity of $\pi_1$). We also have $\rho_Z(V\cap U_x)\leq \rho(X)+\rho_{dY}(U\cap U_y)\leq \rho(X)+\gb<\gl$.

If, on the other hand, $U\cup\{y\}$ is open in $Y$, then $\pi_1^{-1}(U)\cup\{x\}$ is open in $Z$, by the continuity of $\pi_1$. Hence, so is $V\cup \{x\}=W\cup \pi_1^{-1}(U)\cup\{x\}$.
\ep

Consider now two spaces $X=[1,\gl]$ and $Y=[1,\mu]$ equipped with the interval topologies. Notice that since $X$ is compact there is an ordinal $\ga\in X$ whose rank is maximal. Then $\rho(X)=r(\ga)+1\in\Suc$.  Let $X'$ and $Y'$ be any $\ell$-maximal $\ell$-extensions of $X$ and $Y$, respectively. Combining the previous two lemmas we obtain the following corollary.
\bc
$X'\otimes_d Y'$ is an $\ell$-maximal $\ell$-extension of $[1,\gl\mu]$ taken with the interval topology.
\ec

Next, we investigate how $d$-product topology behaves w.r.t.\ the plus operation, for the case of $\ell$-maximal spaces.

\bl \label{prod-plus} Suppose $X$ and $Y$ are $\ell$-maximal and $\rho(X)\in \Suc$. Then $(X\otimes_d Y)^+\simeq (X^+\times \iso(Y))\sqcup (dY)^+$. \el

Here $\sqcup$ denotes the topological sum and $\iso(Y)$ comes with the discrete topology. Also notice that $X^+\times \iso(Y)$ is homeomorphic to $Z_0^+$, and that $(dY)^+$ is homeomorphic to the restriction of $Y^+$ to the set $dY$. (Any set $dA$ on $Y$ is contained in $dY$.)

\bp Let $Z=X\otimes_d Y$ and let $W$ denote $(X^+\times \iso(Y))\sqcup (dY)^+$. We can assume that $Z$ and $W$ have the same underlying set. By Lemma \ref{lmax-plus} the topology of $W$ is generated by sets of the form
\ben
\item
$i_y(V)$, where $y\in\iso(Y)$, $V\in \tau_X$ or $V=d^{\ga+1}X$ with $\ga<\rho(X)$;
\item $\pi_1^{-1}(U\cap dY)$ for $U\in\tau_{Y}$ and $\pi_1^{-1}(d^{\gb+1}Y)$ with $\gb<\rho(Y)$.
\een

To prove the inclusion of $\tau_W$ into $\tau_Z^+$ we check that all these basic open sets are open in $Z^+$.

If $V\in\tau_X$ then $i_y(V)\in\tau_Z$, hence it is open in $Z^+$. If $V=d^{\ga+1}X$ then $i_y(V)=X_y\cap d^{\ga+1}Z$, which is open in $Z^+$ as the intersection of two open sets. If $U\in \tau_Y$, then $\pi_1^{-1}(U\cap dY)=\pi_1^{-1}(U)\cap Z_1$ is open in $Z^+$. In fact, $Z_1=d^{\rho(X)}Z$ is open in $Z^+$, since $\rho(X)\in\Suc$. If $U=\pi_1^{-1}(d^{\gb+1}Y)$, then $U=d_Z^{\gb+1}Z_1=d_Z^{\rho(X)+\gb+1}Z$ which is open in $Z^+$.

Now we check that $\tau_Z^+$ is included in $\tau_W$. Since $X\oplus_d Y$ is $\ell$-maximal, $\tau_Z^+$ is generated by $\tau_Z$ and sets of the form $d^{\ga+1}Z$ for $\ga<\rho(Z)$. By Lemma \ref{rankl}
$$d^{\ga+1}Z=\begin{cases} d^{\ga+1}Z_0 \cup Z_1, & \text{if $\ga<\rho(X)$} \\
\pi_1^{-1}(d^{\gb+1}Y), & \text{if $\ga=\rho(X)+\gb$.} \end{cases}
$$
In both cases it is clearly open in $W$. On the other hand, open sets in $Z$ are generated by $i_y(V)$ with $V\in \tau_X$, in which case we are done, and $\pi_1^{-1}(U)$ with $U\in \tau_Y$. Let $U_0:=U\cap \iso(Y)$ and $U_1:= U\cap dY$. Notice that $\pi_1^{-1}(U_0)=\bigcup_{y\in U_0} X_y$ is open in $Z_0$, and hence in $W$, whereas $\pi_1^{-1}(U_1)$ is open in $dY$, hence in $(dY)^+$ and $W$. Hence, $\tau_Z$ is included in $\tau_W$ and we are done.
\ep

\section{Some operations on lme-spaces}

Recall that $(X,\tau_0,\dots,\tau_n)$ is an \emph{lme-space} based on a scattered topology $\tau$ if
$\tau_0=\tau'$ and $\tau_{i+1}=(\tau_i^+)'$, for each $i<n$, where $\sigma'$ denotes any $\ell$-maximal $\ell$-extension of $\sigma$. Obviously, any such space is a GLP$_n$-space. We call $(X,\tau_0,\dots,\tau_n)$ an \emph{ordinal lme-space} if $X$ is an ordinal and $\tau$ is the interval topology on $X$. We specify two constructions on lme-spaces.

First, we extend the operation of $d$-product to GLP-spaces.

\bd Suppose $(X,\tau_0,\dots,\tau_n)$ and $(Y,\gs_0,\dots,\gs_n)$ are two GLP$_n$-spaces. Let $(Z,\theta_0)$ be the $d$-product $(X,\tau_0)\otimes_d (Y,\gs_0)$. For each $i=1,\dots, n$ we specify a topology $\theta_i$ on $Z$ as the sum of the topologies $\tau_i$ on $X_y$, for each $y\in \iso(Y)$, and of $\gs_i$ on $dY$, where $\iso(Y)$ and $dY$ refer to the space $(Y,\gs_0)$. In other words, $\theta_i$ consists of the sets of the form $$\bigcup_{y\in \iso(Y)} i_y(U_y)\cup \pi_1^{-1}(V\cap dY)$$ where $U_y\subseteq X$, $U_y\in\tau_i$ and $V\in\gs_i$. We note that the functions  $\pi_0:(Z_0,\theta_i\rst Z_0)\tto (X,\tau_i)$ and $\pi_1:(Z_1,\theta_i\rst Z_1)\tto (dY,\gs_i\rst dY)$ are $d$-maps, for $i=1,\dots,n$.
\ed

\bl $(Z,\theta_0,\theta_1,\dots,\theta_n)$ is a GLP$_n$-space. \el

\bp We make use of the fact that the plus operation on topologies distributes over topological sums. Hence, $\theta_i^+\subseteq \theta_{i+1}$ on $Z$, for all $i=1,\dots,n-1$. Thus, we only have to show that $\theta_0^+\subseteq \theta_1$.

Consider any $A\subseteq Z$. By Lemma \ref{dop} $$d_Z(A)=\pi_1^{-1}(d_Y\{y:A\cap X_y\neq \emptyset\})\cup d_{Z_0}(A\cap Z_0).$$ In fact, any $x\in d_Z(A)\cap Z_0$ must belong to $d_{Z_0}(A\cap Z_0)$, since $Z_0$ is open in $Z$, hence the claim. However, both $\pi_1^{-1}(d_Y\{y:A\cap X_y\neq \emptyset\})$ and $d_{Z_0}(A\cap Z_0)$ are open in $\theta_1$. This is because $d_Y\{y:A\cap X_y\neq \emptyset\}$ is open in $(dY,\gs_1)$ and $d_{Z_0}(A\cap Z_0)$ is open in $(Z_0,\theta_1)$.
\ep

\bl Suppose $(X,\tau_0,\dots,\tau_n)$ and $(Y,\gs_0,\dots,\gs_n)$ are lme-spaces based on $\tau$ and $\gs$, respectively, such that both $\rho(X,\tau)$ and $\rho(Y,\gs)$ are successor ordinals. Then $X\otimes_d Y$ is an lme-space based on $(X,\tau)\otimes_d (Y,\gs)$. Moreover, $\rho((X,\tau)\otimes_d (Y,\gs))$ is a successor ordinal.
\el

\bp Let $Z=(Z,\theta_0,\dots,\theta_n)$ denote $X\otimes_d Y$. The fact that $(Z,\theta_0)$ is an $\ell$-maximal $\ell$-extension of $(X,\tau)\otimes_d (Y,\gs)$ follows from Lemmas \ref{prod-lex} and \ref{prod-lexm}.

We show that $(Z,\theta_1)$ is an $\ell$-maximal $\ell$-extension of $(Z,\theta_0^+)$. By Lemma~\ref{prod-plus} $$(Z,\theta_0^+)\simeq ((X,\tau_0^+)\times \iso(Y))\sqcup (dY,\gs_0^+).$$ On the other hand, by definition, $$(Z,\theta_1)\simeq ((X,\tau_1)\times \iso(Y))\sqcup (dY,\gs_1).$$ We have that $(X,\tau_1)$ is an $\ell$-maximal $\ell$-extension of $(X,\tau_0^+)$ and $(dY,\gs_1)$ that of $(dY,\gs_0^+)$. This relation then holds for the respective topological sums.

Finally, we remark that $(Z,\theta_{i+1})$ is an $\ell$-maximal $\ell$-extension of $(Z,\theta_i^+)$, for $i=1,\dots, n$, because $(X,\tau_{i+1})$ is an $\ell$-maximal $\ell$-extension of $(X,\tau_i^+)$ and $(dY,\gs_{i+1})$ is an $\ell$-maximal $\ell$-extension of $(dY,\gs_i^+)$. These relations then must also hold for the respective topological sums.
\ep

\bc
Let $X$ and $Y$ be ordinal lme-spaces on $[1,\gl]$ and $[1,\mu]$, respectively. Then $X\otimes_d Y$ is an ordinal lme-space on $[1,\gl\mu]$.
\ec

We are going to introduce another key operation on lme-spaces called \emph{lifting}. Before doing it we state a simple `pullback' lemma.

\bl \label{lme-pull}
Let $(X,\tau_0,\dots,\tau_n)$ be an lme-space based on $\tau$, and let $h:(Y,\gs)\to (X,\tau)$ be a $d$-map. Then there is an lme-space $(Y,\gs_0,\dots,\gs_n)$ based on $\gs$ such that $h:(Y,\gs_i)\to (X,\tau_i)$ is a $d$-map, for each $i\leq n$.
\el

\bp This statement is proved by a repeated application of Lemmas \ref{pullback} and \ref{lift+} as indicated in the following diagram.

$$
\xymatrix{(Y,\gs)  \ar[d]_d \ar@{.>}[r]_{lm} & (Y,\gs_{0})  \ar@{.>}[d]_d & (Y,\gs_{0}^{+}) \ar@{.>}[d]_d  \ar@{.>}[r]_{lm} & (Y,\gs_{1}) \ar@{.>}[d]_d &  \dots \\
(X,\tau) \ar[r]^{lm} & (X,\tau_0) & (X,\tau_0^+) \ar[r]^{lm} & (X,\tau_1) & \dots }
$$

\medskip
Here, the arrows labeled by `$d$' indicate $d$-maps; the arrows labeled by `$lm$' indicate $\ell$-maximal $\ell$-extensions. Dotted arrows are being proved to exist given the rest. Thus, the two squares represent the first two applications of Lemma \ref{pullback}, and the transition from the right vertical arrow of the first square to the left vertical arrow of the second one is an application of Lemma  \ref{lift+}.
\ep

\bl[lifting] \label{lme-lift}
Suppose $X=([0,\gl],\tau_1,\dots,\tau_n)$ is an ordinal lme-space. Then there is an ordinal lme-space $Y=([1,\gw^\gl],\gs_0,\gs_1,\dots,\gs_n)$ such that
$r:([1,\gw^\gl],\gs_i)\tto ([0,\gl],\tau_i)$ is a $d$-map, for each $i=1,\dots,n$.
\el

Such an $Y$ can be called a \emph{lifting} of the space $X$, since it is similar to $X$ w.r.t.\ higher topologies (starting from the second one rather than the first).

\bp
Topology $\gs_0$, being an $\ell$-maximal $\ell$-extension of the order topology, has the same rank function.
Therefore, $r:([1,\gw^\gl],\gs_0) \tto ([0,\gl],\tau_{\leftarrow})$ is a $d$-map. By Lemma \ref{lift+} we obtain that $r:([1,\gw^\gl],\gs_0^+) \tto ([0,\gl],\tau_{<})$ is a $d$-map, as well. Since $\tau_1$ is an $\ell$-maximal $\ell$-extension of the order topology, we are now in a position to apply Lemma \ref{lme-pull}. So, we obtain an lme-space $([1,\gw^\gl],\gs_1,\dots,\gs_n)$ based on $\gs_0^+$ such that $r:([1,\gw^\gl],\gs_i)\tto ([0,\gl],\tau_i)$ is a $d$-map, for each $i=1,\dots,n$. It follows that $Y=([1,\gw^\gl],\gs_0,\gs_1,\dots,\gs_n)$ is as required.
\ep

\section{Proof of main lemma}
Now we provide the key construction proving Lemma~\ref{l:key construction} above.

\bp
For each $J_n$-tree $(T,R_0,\dots,R_n)$ with a root $a$ we are going to build an ordinal lme-space  $X=([1,\gl],\tau_0,\dots,\tau_n)$ and a $J_n$-morphism $f:X\tto T$ such that $f^{-1}(a)=\{\gl\}$. Such $J_n$-morphisms will be called \emph{suitable}. The construction goes by induction on $n$ with a subordinate induction on the $R_0$-height  of $T$, which is denoted $ht_0(T)$.

\bigskip
If $n=0$ we let $\tau_0$ be the interval topology and notice that on any $\gl<\gw^\gw$ this topology is $\ell$-maximal (since there are no points of limit rank). From the topological completeness proofs for the G\"odel--L\"ob logic it is known (see \cite{BEG05}) that there is an ordinal $\gl<\gw^\gw$ and a suitable $d$-map from $[1,\gl]$ onto $(T,R_0)$. This map is constructed by induction on $ht_0(T)$.

If $ht_0(T)=0$, then $T$ consists of a single point $a$. We put $\gl=1$ and $f(1)=a$. If $ht_0(T)=m>0$ let $a_1,\dots,a_l$ be the children of the root $a$, and let $T_i$ denote the subtree generated by $a_i$, for $i\leq l$. By the induction hypothesis, there are ordinals $\kappa_1,\dots,\kappa_l$ and suitable $d$-maps $g_i:[1,\kappa_i]\tto T_i$, for each $i=1,\dots,l$. Let $\kappa:=\kappa_1+\cdots+\kappa_l$, then $[1,\kappa]$ can be identified with the topological sum  $\bigsqcup_{i=1}^l [1,\kappa_i]$. Let $g:[1,\kappa]\tto \bigsqcup_{i=1}^l T_i$ be defined by
$$g(\ga):=g_i(\gb),\ \text{if $\ga=\kappa_1+\cdots+\kappa_{i-1}+\gb$, $\gb\in[1,\kappa_i]$}.$$ Then $g$ is clearly a $d$-map.

We now let $\gl:=\kappa\gw$ and let $f:[1,\gl]\tto T$ be defined by
$$
f(\ga):=
\begin{cases}
g(\gb), &  \text{if $\ga=\gl n+\gb$ where $n<\gw$, $\gb\in[1,\kappa]$}, \\
a, &  \text{if $\ga=\gl$}.
\end{cases}
$$
It is then easy to verify that $f$ is, indeed, a suitable $d$-map. This accounts for the case $n=0$.

\bigskip
For the induction step suppose the lemma is true for each $J_k$-tree with $k< n$. Let $T=(T, R_0, \dots, R_n)$ be an $J_{n}$-tree with the root $a$. We prove our claim by induction on the $R_0$-depth of $T$.

\medskip
\textsc{Case 1:} $ht_0(T)=0$, in other words $R_0=\emptyset$. Let $T_1:=(T,R_1,\dots,R_n)$. By the induction hypothesis there is a suitable $J_{n-1}$-morphism $f_1:X_1\tto T_1$ where $X_1=([1,\gl_1],\tau_1,
\dots,\tau_n)$. We note that $X_1$ is isomorphic to $([0,\mu],\tau_1, \dots,\tau_n)$, for some $\mu$ (obviously, $\mu=\gl_1$ if $\gl_1$ is infinite). By the Lifting lemma there is an ordinal lme-space $X=([1,\gl],\gs_0,\gs_1,\dots,\gs_n)$ such that $\gl=\gw^{\mu}$ and $$r:([1,\gl],\gs_i)\tto ([0,\mu],\tau_i)$$ is a $d$-map, for each $i\in [1,n]$. It follows that $f:=r\circ f_1$ is a suitable $J_n$-morphism. In fact, it is immediate that conditions $(j_1)$, $(j_2)$ are met and that $(j_3),(j_4)$ are satisfied for each $k\geq 1$. Let us consider $(j_3)$ for $k=0$.

Since $R_0$ is empty, the only $1$-hereditary root of $T$ is in fact the unique $0$-hereditary root $a$, thus $R^*_0(a)=T\setminus \{a\}$. Then clearly  $f^{-1}(R_0^*(a))=[1,\gl)$ and $f^{-1}(R_0^*(a)\cup\{a\})=[1,\gl]$, both of which are $\gs_0$-open. Thus $(j_3)$ is met.

Condition $(j_4)$ for $k=0$ boils down to the fact that $f^{-1}(a)$ is discrete. However, $f^{-1}(a)$ is the  singleton $\{\gl\}$. Thus $(j_4)$ is also met and
$f:X\tto T$ is the required $J_n$-morphism.

\medskip
\textsc{Case 2:} $ht_0(T)=m>0$. Let $a_1,\dots,a_l$ be the immediate $R_0$-successors of $a$ which are hereditary $1$-roots. Denote $T_i=\{a_i\}\cup R_0^*(a_i)$ for $i\in[1,l]$ and $T_0=\{a\}\cup R_1^*(a)$. Note that $T=\bigcup_{i=0}^l T_i$. Furthermore, for each $i\in[1,l]$ the subframe $T_i$ of $T$ is a $J_{n}$-tree of $R_0$-depth less than $m$. By the induction hypothesis there exist ordinal lme-spaces $S_i=([1,\kappa_i],\xi^i_0,\dots,\xi^i_{n})$ and suitable $J_{n}$-morphisms $g_i:S_i\to T_i$. Let $\kappa:=\kappa_1+\dots+\kappa_l$, then $[1,\kappa]$ can be identified with the disjoint union  $\bigsqcup_{i=1}^l [1,\kappa_i]$. Let $\xi_0,\dots,\xi_{n}$ be the topologies of the corresponding topological sum, that is, $\xi_j=\bigsqcup_{i=1}^l \xi^i_j$, and let $g:[1,\kappa]\to \bigsqcup_{i=1}^l T_i$ be the disjoint union of $g_i$, i.e. $g=\bigsqcup_{i=1}^l g_i$. Notice that $\bigsqcup_{i=1}^l T_i$ is identified with $R_0(a)$. It is easy to see that $X=([1,\kappa],\xi_1,\dots,\xi_{n})$ is an ordinal lme-space and that $g:[1,\kappa]\to R_0(a)$ is a $J_n$-morphism.

Now consider the 1-sheet $(T_0,R_1,\dots,R_n)$. By the induction hypothesis (for $n$) there is an ordinal lme-space $Y_0=([1,\gl_0],\tau_1,\dots,\tau_n)$ and a suitable $J_{n-1}$-morphism $g_0:Y_0\tto T_0$. Let $Y=([1,\gw^{\gl_0}],\gs_0,\gs_1,\dots,\gs_n)$ be an ordinal lme-space defined as in \textsc{Case 1} and
let $h:Y\tto (T_0,\emptyset,R_1,\dots,R_n)$ be the corresponding suitable $J_n$-morphism.

We now consider the $d$-product $Z:=X\otimes_d Y$ of these ordinal lme-spaces. Note that $\iso(Y)=\{\ga+1:\ga<\kappa_0\}$ and $dY=\Lim\cap [1,\kappa_0]$. Hence, we can identify $Z$ with an ordinal lme-space $([1,\gl],\theta_0,\dots,\theta_n)$ where $\gl:=\kappa \cdot \gw^{\gl_0}$, $X_{\ga+1}=[\kappa\ga+1,\kappa(\ga+1)]$, for all $\ga<\kappa_0:=\gw^{\gl_0}$. Hence, $Z_0=\bigsqcup_{\ga<\kappa_0}X_{\ga+1}$ and $Z_1=\{\kappa\gl:\gl\in\Lim, \gl\leq\kappa_0\}$.
The associated projection maps $\pi_0:Z_0\tto X$ and $\pi_1:Z\tto Y$ are defined by formulas $\pi_1(\kappa\gl)=\gl$ and $\pi_0(\kappa\ga+\gb)=\gb$, where $\gl\in\Lim$, $\gl\leq\kappa_0$, $\gb\in [1,\kappa]$, $\ga<\kappa_0$.

We define the required $J_n$-morphism $f:Z\tto T$ as follows:
$$
f(z):= \begin{cases} g(\pi_0(z)), & \text{if $z\in Z_0$}, \\
h(\pi_1(z)), & \text{if $z\in Z_1$.}
\end{cases}
$$

We have to check that $f$ satisfies $(j_1)$--$(j_4)$. Recall that for $k\geq 1$ the space $(Z,\theta_k)$ is homeomorphic to the topological sum of $Z_0\simeq \bigsqcup_{\ga<\kappa_0} (X,\xi_k)$ and $Z_1\simeq (Y,\gs_k)$. Then both $\pi_0:(Z_0,\theta_k{\rst} Z_0)\tto (X,\xi_k)$ and $\pi_1:(Z_1,\theta_k{\rst} Z_1)\tto (Y,\gs_k)$ are $d$-maps. Since both $g$ and $h$ are $J_n$-morphisms, it follows that conditions $(j_1)$--$(j_4)$ are satisfied for all $k\geq 1$. We must only check  $(j_2)$--$(j_4)$ for $k=0$.

Recall that the topology $\theta_0$ on a $d$-product $X\otimes_d Y$ is generated by the base of open sets $\{i_y(V):V\in\tau_X, y\in \iso(Y)\}$ and $\{\pi_1^{-1}(U):U\in\tau_Y\}$. Hence, in order to check $(j_2)$ it is sufficient to show that the image under $f$ of any such basic open set is open. Since $i_y(V)\subseteq Z_0$ and $\pi_0(i_y(V))=V$ we obtain that $f(i_y(V))=g(V)$ is open ($g$ is a $J_n$-morphism). On the other hand, if $U$ is non-empty, then $f(\pi_1^{-1}(U))=h(U)\cup g(X)= h(U)\cup R_0(a)$. This holds because every nonempty open subset of $Y$, in particular $U$, has a point $y$ of rank $0$. Then $X_y\subseteq\pi_1^{-1}(U)$ and hence $f(\pi_1^{-1}(U))\supseteq f(X_y)=g(X)$. Clearly, both $h(U)$ and $R_0(a)$ are open in $T$. Hence, $f$ satisfies $(j_2)$.

Condition $(j_3)$ follows from the fact that both $\pi_0$ and $\pi_1$ are continuous. Indeed, if $w$ is a hereditary 1-root of $T$, then either $w=a$ or $w\in R_0(a)$. In the former case $R^*_0(w)=T\setminus\{a\}$ and hence $f^{-1}(R^*_0(w))=[1,\gl)$ is open. Similarly, $f^{-1}(R^*_0(w)\cup \{w\})= Z$ is open.

If $w\in R_0(a)$ then both $R^*_0(w)$ and $R^*_0(w)\cup\{w\}$ are contained in $R_0(a)$. Since $g$ is a $J_n$-morphism, $g^{-1}(R^*_0(w))$ is open. Then $f^{-1}(R^*_0(w))=\pi_0^{-1}(g^{-1}(R^*_0(w))$ is open, by the continuity of $\pi_0$. The argument for $R^*_0(w)\cup\{w\}$ is similar. Hence, condition $(j_3)$ is met.

To check condition $(j_4)$ assume $w$ is a hereditary 1-root of $T$. If $w=a$ then $f^{-1}(w)$ is the singleton $\{\gl\}$. If $w\in R_0(a)$ then $g^{-1}(w)$ is discrete as a subspace of $X$, since $g$ is a $J_n$-morphism. We know that $\pi_0:Z_0\tto X$ is both continuous and pointwise discrete. Hence, $f^{-1}(w)=\pi_0^{-1}(g^{-1}(w))$ is discrete in $Z_0$ and thereby in $Z$ ($Z_0$ is open in $Z$). This shows $(j_4)$.

Thus, we have checked that $f:Z\tto T$ is a suitable $J_n$-morphism, which completes the proof of Lemma \ref{l:key construction} and thereby of Theorem \ref{t:main-1}.
\ep

\appendix
\section{Appendix: Proof of Proposition \ref{neighbor}.}

The correspondence between Magari frames and scattered topological spaces is essentially due to Esakia.
A frame $(X,\gd)$ is called a \emph{Magari frame} if it satisfies the following identities, for any $A,B\subseteq X$:
\begin{enumr}
\item $\gd(A\cup B) = \gd A\cup\gd B$; \quad $\gd\emptyset = \emptyset$;
\item $\gd A = \gd (A\setminus \gd A)$.
\end{enumr}
It is well-known and easy to see that $(X,\gd)$ is Magari iff $(X,\gd)$ validates the axioms of G\"odel--L\"ob logic $\GL$ (corresponding to Axioms (i)--(iii) of $\GLP$). We notice that any such operator $\gd$ is monotone, that is, $A\subseteq B$ implies $\gd A\subseteq \gd B$. In addition, $\gd\gd A\subseteq \gd A$ holds in any Magari frame, since the formula $\Diamond\Diamond p\to \Diamond p$ is a theorem of $\GL$.

\bl
If $(X,\tau)$ is a scattered topological space then $(X,d_\tau)$ is a Magari frame.
\el

\bp The validity of (i) is obvious, whereas (ii) means that any limit point of $A$ is a limit point of the set $\iso(A)$ of isolated points of $A$. Let $x\in d_\tau A$ and let $U$ be an open neighborhood of $x$. $U\cap A\setminus\{x\}$ is not empty, hence it has an isolated point $y$. Then $y\in\iso(A)$ as well.
\ep

Suppose $(X,\tau_0,\tau_1,\dots)$ is a GLP-space.
To prove Part (i) of Proposition \ref{neighbor} observe that Axioms (i)--(iii) of $\GLP$ are satisfied in $(X,d_0,d_1,\dots)$ by the previous corollary. Axiom (iv) is clearly valid since $\tau_n\subseteq \tau_{n+1}$.

To check Axiom (v) consider a set of the form $d_n(A)$. Since $X$ is a GLP-space, $d_n(A)$ is open in $\tau_{n+1}$. Hence, every $x\in d_n(A)$ cannot be a $\tau_{n+1}$-limit point of $X\setminus d_n(A)$, that is, $x\in \tilde{d}_{n+1}(d_n A)$. In other words, $d_n(A)\subseteq \tilde{d}_{n+1}(d_n A)$, for any $A$, that is, Axiom (v) is valid.

\bigskip
To prove Part (ii) of Proposition \ref{neighbor} we first remark that, if $(X,\gd)$ is a Magari frame, then the operator $c(A):= A\cup \gd A$ satisfies the Kuratowski axioms of the topological closure. This defines a topology on $X$ in which any set $A$ is closed iff $A=c(A)$ iff $\gd A\subseteq A$. (Alternatively, one can check that the collection of all sets $U$ satisfying $U\subseteq \tilde\gd U$ is a topology.)

\bl Suppose $(X,\gd)$ is Magari. Then, for all $x\in X$,
\begin{enumr}
\item $x\notin \gd(\{x\})$;
\item $x\in \gd A \iff x\in \gd(A\setminus\{x\})$.
\end{enumr}
\el

\bp (i) By Axiom (iii) $\gd\{x\}\subseteq \gd(\{x\}\setminus \gd\{x\})$. If $x\in \gd\{x\}$ then $\gd(\{x\}\setminus \gd\{x\})\supseteq \gd(\{x\}\setminus\{x\})=\gd\emptyset=\emptyset$. Hence, $\gd\{x\}=\emptyset$, a contradiction.

(ii) $x\in \gd A$ implies $x\in \gd ((A\setminus\{x\})\cup \{x\})= \gd(A\setminus\{x\})\cup \gd\{x\}$. By (i), $x\notin\gd\{x\}$, hence $x\in \gd(A\setminus\{x\})$. The other implication follows from the monotonicity of $\gd$.
\ep

\bl Suppose $(X,\gd)$ is Magari and $\tau$ is the associated topology.
Then $\gd=d_\tau$.
\el

\bp Let $d=d_\tau$; we show that, for any set $A\subseteq X$, $dA=\gd A$. Notice that, for any $B$, $cB=dB\cup B=\gd B\cup B$.

Assume $x\in \gd A$ then $$x\in \gd(A\setminus\{x\})\subseteq c(A\setminus\{x\})\subseteq d(A\setminus \{x\})\cup (A\setminus\{x\}).$$ Since $x\notin A\setminus\{x\}$, we obtain $x\in d(A\setminus \{x\})$. By the monotonicity of $d$, $x\in dA $.

Similarly, if $x\in dA$ then $x\in d(A\setminus\{x\})$. Hence, $$x\in c(A\setminus\{x\})=\gd(A\setminus\{x\})\cup (A\setminus \{x\}).$$
Since $x\notin A\setminus\{x\}$ we obtain $x\in \gd A$.
\ep

\bl Suppose $(X,\gd)$ is Magari and $\tau$ is the associated topology.
Then $(X,\tau)$ is scattered.
\el

\bp Since $\gd$ is L\"ob we know that $\gd=d_\tau$. We show that any nonempty subspace $A\subseteq X$ has an isolated point.

Suppose not, then $\iso(A)=A\setminus \gd A=\emptyset$. Then $\gd A=\gd (A\setminus \gd A)=\gd\emptyset=\emptyset$. Then $A=A\setminus\gd A= \emptyset$.
\ep

Now we prove Part (ii). Let $(X,\gd_0,\gd_1,\dots)$ be a neighborhood frame satisfying $\GLP$. Then each of the frames $(X,\gd_n)$ is Magari, hence it defines a scattered topology $\tau_n$ on $X$ for which $\gd_n= d_{\tau_n}$. Recall that $U\in\tau_n$ iff $U\subseteq \tilde\delta_n(U)$. We only have to show that the last two conditions of a GLP-space are met.

Suppose $U\in\tau_n$, then $U\subseteq\tilde\delta_n(U)\subseteq\tilde\delta_{n+1}(U)$ by Axiom (iv). Hence, $U\in\tau_{n+1}$. Thus, $\tau_n\subseteq \tau_{n+1}$.

Similarly, by Axiom (v) for any set $A$ we have $\delta_n(A)\subseteq \tilde\delta_{n+1}(\delta_n(A))$. Hence, $d_{\tau_n}(A)=\delta_n(A)\in\tau_{n+1}$. Thus, $(X,\tau_0,\tau_1,\dots)$ is a GLP-space.


\end{document}